\documentclass[11pt]{amsart}
\usepackage{amsfonts,amsmath,amsthm,amssymb,hyperref,fancyhdr,float,caption,mathtools}
\usepackage{wasysym}
\usepackage{hyperref,fancyhdr,tikz,mathtools}
\usepackage{enumerate}
\usepackage{latexsym}
\usepackage{amsfonts}
\usepackage[utf8]{inputenc}
\RequirePackage{thmtools}
\RequirePackage[all, pdf]{xy}
\RequirePackage{enumitem}

\DeclareMathOperator{\Regular}{Reg}

\DeclareMathOperator{\Rls}{Rls}
\DeclareMathOperator{\undsp}{sp}

\DeclareMathOperator{\Object}{Obj}
\DeclareMathOperator{\Morphism}{Mor}

\DeclareMathOperator{\tPre}{Pre}

\DeclareMathOperator{\tSuc}{Suc}
\DeclareMathOperator{\tsuc}{Ch}
\DeclareMathOperator{\tssuc}{Gch}
\DeclareMathOperator{\tRoot}{Rt}

\DeclareMathOperator{\tLeaf}{Lf}
\DeclareMathOperator{\tdepth}{dpt}
\DeclareMathOperator{\theight}{hgt}
\DeclareMathOperator{\tlevel}{Lv}

\newcommand*{\fibder}{\mathop{}\!\lb{d}}
\newcommand*{\rint}{\textsc{r}\!\!\!\!\!\int}
\newcommand*{\pcspt}[1]{\Pa{#1}_\cspt}
\newcommand*{\FVform}{\Omega_{F, \mathrm{vol}}}
\newcommand*{\Conn}{\operatorname{Conn}}

\newcommand*{\cont}{\mathrm{C}}
\newcommand*{\smth}{\mathrm{C}^\infty}
\newcommand*{\lBundle}{\textup{FilBund}}
\newcommand*{\ConnlBundle}{\textup{CFilBund}}
\newcommand*{\lb}[1]{\mathbf{#1}}

\newcommand*{\Tse}{\Theta}
\newcommand*{\Shta}{\Xi}
\newcommand*{\bfomega}{\boldsymbol{\omega}}
\newcommand*{\bfiota}{\boldsymbol{\iota}}
\newcommand*{\bfmu}{\boldsymbol{\mu}}

\makeatletter
\newif\if@tikz
\@tikzfalse
\if@tikz
\usepackage{tikzheader}
\usetikzlibrary{external}
\providecommand{\tikzsetnextfilename}[1]{}
\newcommand{\inputfigure}[3]{%
\tikzsetnextfilename{#1}
\begin{figure}[htb]
  \centering
  \input{fig/#1.tikz}
  \caption{#3} \label{fig:#2}
\end{figure}%
}
\else
\usepackage{graphicx}
\newcommand{\inputfigure}[3]{%
\begin{figure}[htb]
  \centering
  \includegraphics{figext/#1}
  \caption{#3} \label{fig:#2}
\end{figure}%
}
\fi

\makeatother

\declaretheorem[numberwithin=section]{theorem}
\declaretheorem[sibling=theorem]{lemma}
\declaretheorem[sibling=theorem]{corollary}
\declaretheorem[sibling=theorem]{proposition}

\declaretheorem[style=definition, numberwithin=section]{definition}

\numberwithin{equation}{section}

\RequirePackage[capitalise]{cleveref}
\crefname{theorem}{Theorem}{Theorems}
\crefname{maintheorem}{Main Theorem}{Main Theorems}
\crefname{lemma}{Lemma}{Lemmas}
\crefname{corollary}{Corollary}{Corollaries}
\crefname{definition}{Definition}{Definitions}
\crefname{example}{Example}{Examples}
\crefname{remark}{Remark}{Remarks}
\crefname{question}{Question}{Questions}
\crefname{enumi}{}{}
\crefname{equation}{}{}
\crefname{figure}{Figure}{Figures}
\crefname{table}{Table}{Tables}

\makeatletter
\newif\if@refcheck
\@refcheckfalse
\if@refcheck
\usepackage{refcheck}
\newcommand{\refcheckize}[1]{%
  \expandafter\let\csname @@\string#1\endcsname#1%
  \expandafter\DeclareRobustCommand\csname relax\string#1\endcsname[1]{%
    \csname @@\string#1\endcsname{##1}\wrtusdrf{##1}}%
  \expandafter\let\expandafter#1\csname relax\string#1\endcsname
}
\refcheckize{\cref}
\refcheckize{\Cref}
\fi
\makeatother

\let\originalleft\left
\let\originalright\right
\renewcommand*{\left}{\mathopen{}\mathclose\bgroup\originalleft}
\renewcommand*{\right}{\aftergroup\egroup\originalright}
\renewcommand*{\geq}{\geqslant}
\renewcommand*{\leq}{\leqslant}

\let\emptyset\varnothing

\DeclareMathOperator{\support}{supp}
\DeclareMathOperator{\identity}{id}
\DeclareMathOperator{\proj}{pr}
\DeclareMathOperator{\Diff}{Diff}

\newcommand*{\abs}[1]{\left\lvert#1\right\rvert}

\newcommand*{\Set}[1]{\mathchoice{\left\{#1\right\}}{\{#1\}}{\{#1\}}{\{#1\}}}
\newcommand*{\Setby}[2]{\mathchoice{\left\{#1\mathrel{}\middle|\mathrel{}#2\right\}}{\{#1\mid#2\}}{\{#1\mid#2\}}{\{#1\mid#2\}}}
\newcommand*{\Pa}[1]{\mathchoice{\left(#1\right)}{(#1)}{(#1)}{(#1)}}

\newcommand*{\Res}[1]{\mathchoice{\left.\kern-\nulldelimiterspace#1\right|}{#1|}{#1|}{#1|}}

\newcommand*{\der}{\mathop{}\!\mathrm{d}}
\newcommand*{\cspt}{\mathrm{c}}

\newcommand*{\defeq}{\stackrel{\textnormal{def}}{=}}

\newcommand{\R}{\mathbb{R}}

\newcommand{\N}{\mathbb{N}}

\title[Moser stability for volume forms on noncompact fiber bundles]{Moser stability for volume forms\\ on noncompact fiber bundles}
\author{\'Alvaro Pelayo \qquad \qquad Xiudi Tang}

\begin{document}

\begin{abstract}
  We prove a stability result for volume forms on fiber bundles with compact base and noncompact fibers.  
  This generalizes the classical results of Moser and Greene--Shiohama, and recent work by the authors.
\end{abstract}

\maketitle

\section{Introduction} \label{sec:intro}

Moser proved that two volume forms on a compact manifold with equal total integral are diffeomorphic~\cite{MR0182927}. 
This was extended to noncompact manifolds by Greene and Shiohama~\cite{MR542888}, and in~\cite{PeTa18} to smooth families. 
The main result in~\cite{PeTa18} says that if $B$ is a compact manifold, and $M$ is an oriented manifold on which one has two \emph{smooth} families of volume forms $\Set{\omega_p}_{p \in B}$ and $\Set{\tau_p}_{p \in B}$ such that for each $p \in B$, $$\int_M \omega_p = \int_M \tau_p,$$ then there is a \emph{smooth} family of diffeomorphisms $\Set{\varphi_p \colon M \to M}_{p \in B}$ with $\varphi^*_p \omega_p = \tau_p$ for each $p$ provided the following holds: for any connected component $C$ of an end of $M$, either $C$ has infinite volume with respect to both $\Set{\omega_p}_{p \in B}$ and $\Set{\tau_p}_{p \in B}$, or $p \mapsto \int_C \omega_p$ and $p \mapsto \int_C \tau_p$ are continuous and their difference is smooth.  
Our goal is to prove a natural generalization of this result (\cref{thm:main-theorem}) for fiber bundles $\pi \colon M \to B$ with noncompact fiber $F$ for which $M$ is exhausted by some function $f \colon M \to \R$ compatible with the fiber bundle; we will call these \emph{exhausted bundles}. 

While the paper generalizes~\cite{PeTa18},  it is self-contained. 
Next we are going to introduce the key notions of the paper. 
Our main theorem is stated in terms of these notions at the end of the section. 

\subsection{Exhausted and filled (sub)bundles} \label{ssec:fil-exh-bundle}

Throughout this paper manifolds are ($\mathrm{C}^\infty$) smooth without boundary unless otherwise stated. 
Let $F, M, B$ be smooth manifolds, where $M, B$ may have boundaries and $B$ is connected.
Let $\pi \colon M \to B$ be a smooth surjective map. 
Suppose that for every $p \in B$ there exists an open neighborhood $U$ of $p$ in $B$ and a diffeomorphism $\phi \colon \pi^{-1}(U) \to U \times F$ such that $\pi = \phi \circ \proj_1$, where $\proj_1$ is the projection onto the first factor of the product. 
As usual, $(\pi, M, B, F)$ is called a \emph{fiber bundle}, with \emph{underlying space} $M$, \emph{base} $B$, and \emph{fiber} $F$.  
We call $U$ a \emph{trivializing region of $\pi$} and $(U, \phi)$ a \emph{trivialized chart}. 
 
\begin{definition} \label{def:filled-bundle}
  Let $(\pi, M, B, F)$ be a fiber bundle.
  Let $\Set{(U_i, \phi_i)}_{i \in \mathcal{I}}$ be a local trivialization, $f \colon M \to \mathbb{R}$ a smooth function, and $\{h_i \colon F \to \mathbb{R}\}_{i \in \mathcal{I}}$ a family of smooth functions such that $f \circ \phi_i^{-1} = h_i \circ \proj_2$, where $\proj_2$ is the projection onto the second factor of the product.
  We call $\lb{M} \defeq (\pi, M, B, F, f)$ is a \emph{filled bundle}. If $f$ and all the $h_i, i \in \mathcal{I}$, are exhaustion functions, we call $\lb{M}$ an \emph{exhausted bundle}. 
\end{definition}

We call $\Set{(U_i, \phi_i)}_{i \in \mathcal{I}}$ \emph{compatible} with $f$.
We write $\undsp(\lb{M}) \defeq M$.
 
\begin{definition}
  Let $\lb{M} = (\pi, M, B, F, f)$ be a filled bundle, $A$ a submanifold of $M$ with or without boundary, and $\Set{(U_i, \phi_i)}_{i \in \mathcal{I}}$ a local trivialization compatible with $f$.
  We call $A$ a \emph{filled subspace of} $\lb{M}$ with respect to the trivialization if for any $i \in \mathcal{I}$ there is $P_i \subset F$ with $\phi_i(A \cap \pi^{-1}(U_i)) = U_i \times P_i$.
  If $U_i \cap U_j \neq \emptyset$ then there is a diffeomorphism of $F$  induced by a change of charts $\phi_j \circ \phi_i^{-1}$ sending $P_i$ to $P_j$.
  Since $B$ is connected, the $P_i$, $i \in \mathcal{I}$, are diffeomorphic; let $P$ be one of them.
  Then $\lb{A} \defeq (\Res{\pi}_A, A, B, P, \Res{f}_A)$ is a filled bundle, which we call a \emph{filled subbundle} of $\lb{M}$.
  We write $\Res{\lb{M}}_A \defeq \lb{A}$.
\end{definition}

Here are some examples of filled and exhausted bundles:
\begin{enumerate}
  \item
    Let $(\pi, N, B, E)$ be a compact fiber bundle and let $F$ be a noncompact manifold with a smooth function $h \colon F \to \mathbb{R}$.
    Then $\lb{M} = (\pi \circ \proj_1, N \times F, B, E \times F, h \circ \proj_2)$ is a filled bundle which is exhausted if and only if $h$ is an exhaustion function.
  \item
    Let $F = \Setby{(x,y,z) \in \mathbb{R}^3}{x^2 + y^2 = 1, y^2 + z^2 > \frac{1}{4}}$, then $F$ is a noncompact $2$-manifold with $4$ ends.
    Let $\phi \in \Diff(F)$ be the diffeomorphism given by $\phi(x, y, z) = (x, -y, -z)$, switching two ends $z \to +\infty$ and $z \to -\infty$.
    Let $h \colon F \to \mathbb{R}, h(x, y, z) = z^2 + \Pa{(y^2+z^2) - \frac14}^{-1}$, then $h$ is an exhaustion function with the property $h \circ \phi = h$.
    Let $B = \mathbb{S}^1 = (0, 2) / (p \mapsto p+1)$.
    Then define $M = (0, 2) \times_\varphi F$ where $\varphi \colon (1, 2) \times F \to (0, 1) \times F$ is given by $\varphi(p, y) = (p-1, \phi(y))$.
    Let $\pi \colon M \to B$ be the map induced by $\proj_1 \colon (0, 2) \times F \to (0, 2)$, and $f \colon M \to \mathbb{R}$ be the map induced by $h \circ \proj_2 \colon (0, 2) \times F \to \mathbb{R}$.
    Then $\lb{M} = (\pi, M, B, F, f)$ is an oriented exhausted bundle where $M$ has $3$ ends, since $\phi$ is orientation preserving.
  \item
    Let $\mathcal{G}$ be a subgroup of $\mathrm{SO}(n)$.
    Let $E \subset \mathbb{R}^k$ be a noncompact complete submanifold, which is invariant under $\mathcal{G}$.
    Let $u \colon \mathbb{R}^k \to \mathbb{R}$ be a smooth function such that $u \circ \phi = u$ for any $\phi \in \mathcal{G}$.
    Let $F = E \cap \Set{u > 0}$.
    Let $h \colon F \to \mathbb{R}, h(x) = \abs{x}^2 + u(x)^{-1}$, then $h$ is an exhaustion function with the property $h \circ \phi = h$ for any $\phi \in \mathcal{G}$.
    Let $(\pi, M, B, F)$ be any fiber bundle with structure group $\mathcal{G}$ such that $B$ is compact.
    Let $f$ be the unique exhaustion for $M$ such that the transition maps in $\mathcal{G}$ is compatible with $f$, with same $h = h_i$.
    Then $\lb{M} = (\pi, M, B, F, f)$ is an oriented exhausted bundle with noncompact fiber.
\end{enumerate}

\subsection{Releasing a filled bundle} \label{ssec:release}

A diffeomorphism will be constructed in our main theorem to move volumes within each fiber of a filled bundle $\lb{M}$.
To construct it we chop the bundle into subbundles $\lb{A}$, with disconnected fibers. 
Because of this we cannot transfer volumes between connected components. To resolve the issue we introduce a new filled bundle: $\Rls \lb{A}$.

For any topological spaces $X, Y$ let $\Conn X$ be the set of connected components of $X$ and if $\mu \colon X \to Y$ is continuous, let $\Conn \mu \colon \Conn X \to \Conn Y$ be the map sending $C$ to the connected component of $Y$ containing $\mu(C)$. 

\begin{proposition} \label{prop:release}
  Let $\lb{M} = (\pi, M, B, F, f)$ be a filled bundle with $M$ connected. 
  Let $B_M \defeq \coprod_{p \in B} \Conn \pi^{-1}(p)$. 
  Let $(U, \phi)$ be a trivialized chart. 
  Let $\lambda_U \colon \coprod_{p \in U} \Conn \pi^{-1}(p) \to U \times \Conn{\pi^{-1}(U)}$ send a connected component $C$ of $\pi^{-1}(p)$ to $p$ paired with the connected component of $\pi^{-1}(U)$ containing $C$. 
  Endow $B_M$ with the smooth structure for which $\lambda_U$ is a diffeomorphism.
  Then $c_M \colon B_M \to B$ given by $c_M \Pa{\Conn \pi^{-1}(p)} = \Set{p}$, $p \in B$, is a covering. 
  Moreover, $\Rls \pi \colon M \to B_M$ given by $\Res{(\Rls \pi)}_{\pi^{-1}(p)} \defeq \Conn \colon \pi^{-1}(p) \to \Conn \pi^{-1}(p)$, $p \in B$, is a fiber bundle with connected fiber, say $F_M$, and 
  \begin{equation} \label{eq:release-object}
    \Rls \lb{M} \defeq (\Rls \pi, M, B_M, F_M, f)
  \end{equation}
  is a filled bundle.
\end{proposition}

\begin{proof}
  We have a commutative diagram
  \begin{equation*}
    \xymatrixcolsep{4pc}
    \xymatrix{
      \pi^{-1}(U) \ar[r]^-{\pi \times \Conn} \ar[dr]^{\Rls \pi} \ar[d]^{\pi} & U \times \Conn{\pi^{-1}(U)} \\
      U & \coprod\limits_{p \in U} \Conn \pi^{-1}(p) \ar[u]_{\lambda_U} \ar[l]^-{c_M}
    },
  \end{equation*}
  so $c_M$ is a covering, $\Rls \pi$ is smooth and locally trivial, and $(\Rls \pi)^{-1}(p)$ is connected for each $p \in B_M$.
  Since $M$ is connected, so is $B_M = (\Rls \pi)(M)$.
  Hence all fibers of $\Rls \pi$ are diffeomorphic.
\end{proof}  

We call $\Rls \lb{M}$ in \cref{prop:release} the \emph{releasing} of $\lb{M}$.

\subsection{Fiber forms}

Let $\lb{M} = (\pi, M, B, F, f)$ be a filled bundle with oriented fiber.
Let $\iota_p \colon \pi^{-1}(p) \hookrightarrow M$ be the inclusion, $p \in B$.
A \emph{fiber $k$-form} on $\lb{M}$ is a family $\Set{\omega_p}_{p \in B}$ such that $\omega_p$ is a $k$-form on $\pi^{-1}(p)$ and there exists $\omega \in \Omega^k(M)$ with $\omega_p = \iota_p^* \omega$.
We denote $\Set{\omega_p}_{p \in B} \defeq \omega$. A \emph{fiber top-form} is a fiber $(\dim F)$-form. 
A \emph{fiber volume form} is a fiber top-form $\omega$  such that $\omega_p$ is a volume form on $\pi^{-1}(p)$.
Let $\Omega^k_F(\lb{M})$ be the space of fiber $k$-forms on $\lb{M}$. 
The space of compactly supported fiber $k$-forms $\pcspt{\Omega^k_F}(\lb{M})$ on $\lb{M}$ is defined analogously.  
Let $\FVform(\lb{M})$ be the space of fiber volume forms on $\lb{M}$. 

\subsection{Statement of Main Theorem} \label{ssec:main-theorem}
  
Let $\lb{M} = (\pi, M, B, F, f)$ be a filled bundle with oriented fibers.
If $\omega$ is a fiber top-form on $\lb{M}$ and for all $p \in B$ the integral $\int_{\pi^{-1}(p)} \omega_p$ exists (it can be $\pm \infty$), we call the map defined by $(\int_{\lb{M}} \omega)(p)= \int_{\pi^{-1}(p)} \omega_p$ the \emph{fiber integral of $\omega$ on $\lb{M}$}. 

\begin{definition} \label{def:commensurable}
  Two forms $\omega, \tau \in \FVform(\lb{M})$ are \emph{commensurable on a filled subbundle $\lb{A}$ of $\lb{M}$} if their \emph{released fiber integrals} on $\lb{A}$:
  \begin{align*}
    \rint_{\lb{A}} \omega \defeq \int_{\Rls{\lb{A}}} \omega \quad \text{and} \quad \rint_{\lb{A}} \tau \defeq \int_{\Rls{\lb{A}}} \tau \colon B_A &\longrightarrow [-\infty, +\infty] 
  \end{align*}
  exist and are continuous with smooth difference, or are both infinite.
  We say that $\omega, \tau$ are \emph{commensurable} if they are commensurable on the restriction of $\lb{M}$ to every unbounded connected component\footnote{Such restriction is always a filled subbundle, proved in \cref{lem:subbundle-slicing}} of $f^{-1}(\alpha, +\infty)$, $\alpha \in \Regular(f)$.
\end{definition}

Below a diffeomorphism $\varphi$ of $M$ is a \emph{fiber diffeomorphism} if $\pi \circ \varphi = \varphi \circ \pi$.
If $\omega$ is a fiber $k$ form we define $\varphi^* \omega = \Set{\Res{\varphi}_{\pi^{-1}(p)}^* \omega_p}_{p \in B}$. 

\begin{theorem} \label{thm:main-theorem}
  Let $\lb{M}= (\pi, M, B, F, f)$ be a connected exhausted bundle with compact base $B$ and oriented noncompact connected fiber $F$. 
  Then for any commensurable fiber volume forms $\omega, \tau$ on $\lb{M}$ with equal fiber integral, there exists a fiber diffeomorphism $\varphi \colon M \to M$ such that $\varphi^* \omega = \tau$.
\end{theorem}

We conclude with a few remarks:

\begin{enumerate}
  \item
    The following is an interesting problem: give conditions on a fiber bundle so that it admits an exhausted bundle structure. 
  \item
    If the fiber bundle in \cref{thm:main-theorem} is trivial we recover \cite[Theorem 1.1]{PeTa18}. 
    If $B$ is a point, this was proved by Greene and Shiohama~\cite{MR542888}.
  \item
    The proof strategy of \cref{thm:main-theorem} consists of giving the manifold a tree structure, and then constructing in terms of it a global diffeomorphism intertwining the volume forms by glueing. 
    This strategy is analogous to the one adopted in~\cite{PeTa18} but the results we prove do not follow from~\cite{PeTa18}. 
    On the other hand this is no surprise since the main theorem for smooth families can be stated with essentially no preliminaries but for fiber bundles a lot more preparation was required (\cref{ssec:fil-exh-bundle,ssec:release}) to state \cref{thm:main-theorem}.
    The reason was explained in \cref{ssec:release} where the key notion of \emph{releasing a fiber bundle} is given. 
    It is in terms of this notion that we can express the conditions on the integrals over the bundle.    
    The delicate problem has to do with the connectivity of the fibers of fiber bundles not being in general inherited by subbundles (which \emph{cannot} occur for trivial bundles as considered in~\cite{PeTa18}).
  \item
    Understanding the geometry of volume forms is important in classical mechanics, see for instance~\cite{MR2827114}.
\end{enumerate}

\section{Category of filled bundles and release functor} \label{sec:category}

Next we define the categories of filled bundles and connected filled bundles for an important reason: later we cut filled subbundles of the exhausted bundle $\lb{M}$ into subsubbundles, and we will distribute 
volumes of the fibers of subbundles into those of subsubbundles. 
As discussed in \cref{ssec:release}, the nonconnectivity of the fibers forces us to consider the releasing of subbundles and subsubbundles. 
This causes another problem, that their bases are different manifolds. 
Thanks to the functoriality of the release operation (\cref{lem:release-functor}), the bases of subsubbundles are covering spaces of those of subbundles, which makes the distribution of volumes feasible. 

Let $\lBundle$ be the category the objects of which are the filled bundles $\lb{M} = (\pi, M, B, F, f) \in \Object(\lBundle)$ and with morphisms 
from $\lb{M}$ to $\lb{M}' = (\pi', M', B', F', f')$ given by $\bfmu = (\mu, \mu_{B})$, where $\mu$ and $\mu_{B}$ are smooth maps such that
\begin{equation*}
  \xymatrixcolsep{3pc}
  \xymatrix{
    & M \ar[r]^{\pi} \ar[dl]_{f} \ar[d]^{\mu} & B \ar[d]^{\mu_{B}} \\
    \mathbb{R} & M' \ar[r]^{\pi'} \ar[l]_{f'} & B'
  }
\end{equation*}
commutes. 
Denote by $\Morphism(\lb{M}, \lb{M}')$ the space of morphisms from $\lb{M}$ to $\lb{M}'$.

\begin{lemma} \label{lem:release-functor}
  Let $\ConnlBundle$ be the subcategory of $\lBundle$ whose objects have connected underlying spaces.
  Let $(\Rls \pi, M, B_M, F_M, f)$ be as  in \cref{prop:release}. 
  Then there is a functor $\Rls \colon \ConnlBundle \to \ConnlBundle$ such that on objects $\Rls(\pi, M, B, F, f) \defeq (\Rls \pi, M, B_M, F_M, f)$.
\end{lemma}

\begin{proof}
  Let $\lb{M} = (\pi, M, B, F, f)$ and $\lb{M'} = (\pi', M', B', F', f')$ be objects of $\ConnlBundle$. 
  Let $\bfmu = (\mu, \mu_B) \in \Morphism(\lb{M}, \lb{M'})$.
  Let $\Rls \lb{M} = (\Rls \pi, M, B_M, F_M, f)$ and $\Rls \lb{M}' = (\Rls \pi', M', B_{M'}, F_{M'}, f')$ be as  in \cref{prop:release}.
  We are going to define the functor $\Rls$ on morphisms of $\ConnlBundle$.
  Let $\nu \colon B_M \to B_{M'}$ be the unique map defined by the commutative diagram (which also clarifies the relationships among $\nu$ and the maps defined in \cref{prop:release})
  \begin{equation*}
    \xymatrixcolsep{5pc}
    \xymatrix{
      \pi^{-1}(U) \ar[r]^{\mu} \ar@/_5pc/[dd]_{\Rls \pi} \ar[d]^{\pi \times \Conn} & (\pi')^{-1}(U') \ar@/^5pc/[dd]^{\Rls \pi'} \ar[d]_{\pi' \times \Conn} \\
      U \times \Conn \pi^{-1}(U) \ar[r]^{\mu_B \times \Conn \mu} \ar[d]^{\lambda_U} & U' \times \Conn (\pi')^{-1}(U') \ar[d]_{\lambda_{U'}} \\
      \coprod_{p \in U} \Conn \pi^{-1}(p) \ar[r]^{\nu} \ar[d]^{c_M} & \coprod_{p' \in U'} \Conn (\pi')^{-1}(p') \ar[d]_{c_{M'}} \\
      U \ar[r]^{\mu_B} & U'.
    }
  \end{equation*}
  Here $U$ is any open subset of $B$ such that $U$ and $U' \defeq \mu_B(U)$ are trivializing regions of $\pi$ and $\pi'$, respectively.\footnotemark
  \footnotetext{To define $\nu$ we only need the middle rectangle in the diagram but the diagram provides a useful way to keep in mind all maps involved. 
  Also, note that $\nu$ is uniquely defined because $\lambda_{U}$ is a diffeomorphism; it is well defined because the collection of all such $U$ is a base of the topology of $B$, and the definitions of $\nu$ on the preimages of overlapping regions by $c_M$ coincide.}
  Define
  \begin{equation} \label{eq:release-morphism}
    \Rls(\bfmu) \defeq (\mu, \nu).
  \end{equation}
  We have defined $\Rls$ on objects by formula~(\ref{eq:release-object}) and on morphisms by formula~(\ref{eq:release-morphism}).
  One can verify that $\Rls$ assigns the identity map to the identity map and is associative, and therefore $\Rls$ is a functor.
\end{proof}

We call $\Rls$ the \emph{release functor} (the relation between $\lb{M}$ and $\Rls \lb{M}$ is shown in \cref{fig:covering-space}). 

\inputfigure{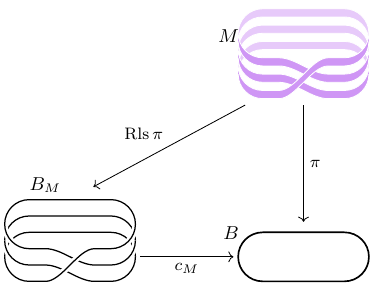}{covering-space}{
$\Rls \lb{M} = (\Rls \pi, M, B_M, F_M, f)$.
}

We apply \cref{lem:release-functor} to the inclusion of bundles:

\begin{corollary} \label{cor:release-subbundle}
  Let $\lb{M} = (\pi, M, B, F, f)$ be a filled bundle. 
  Let $A$ be a filled subspace of $\lb{M}$ and $\lb{A} \defeq \Res{\lb{M}}_A$. 
  Let $\iota \colon A \hookrightarrow M$ be the inclusion and define the morphism $\bfiota \defeq (\iota, \identity_B)$ from $\lb{A}$ to $\lb{M}$.
  Then $\Rls{\bfiota} = (\iota, \kappa)$ where $\kappa \colon B_A \to B_M$ is the unique map induced by the natural map $\Conn{\Res{\pi}_A^{-1}(U)} \to \Conn{\pi^{-1}(U)}$, for any trivializing region $U$ of $\pi$.
\end{corollary}

We call $\bfiota$ in Lemma~\ref{cor:release-subbundle} the \emph{(inclusion) embedding} of $\lb{A}$ into $\lb{M}$. 

\begin{lemma} \label{lem:covering-kappa}
  If the base of $\lb{M}$ is compact and the fiber of $\lb{A}$ has finitely many connected components then $\kappa$ in \cref{cor:release-subbundle} is a covering map between compact spaces.
\end{lemma}

\begin{proof}
  For any open $U \subset B$ whose closure is contained in a trivializing region of $\pi$, $\overline{U} \times \Conn \pi^{-1}(\overline{U})$ consists of finitely many copies of $\overline{U}$, so it is compact.
  Since $B$ has a cover by finitely many such $U$, $B_M$ is the union of finitely many sets diffeomorphic (by $\lambda_U^{-1}$) to $\overline{U} \times \Conn \pi^{-1}(\overline{U})$, so $B_M$ is compact.
  If the fiber of $\lb{A}$ has finitely many connected components, analoguous arguments ensure the compactness of $B_A$.

  Since $B_A$ is compact, $\kappa(B_A)$ is compact.
  But since $\kappa$ is a local diffeomorphism, $\kappa(B_A)$ is open in $B_M$, which means $\kappa$ is surjective, hence a covering map.
\end{proof}

If $\kappa \colon B' \to B$ is a covering space with $B$ connected we denote by $\# \kappa$ the \emph{number of sheets} of $\kappa$, that is, the number of $\kappa^{-1}(p)$ for any $p \in B$ (independent of $p$).

\section{Ingredients for the proof of the main theorem} \label{sec:ingredients}

In this section we prove most of the intermediate statements needed to prove the main theorem. 
The results of the section generalize, but do not follow directly from, the results of \cite[Sections 2 and 3]{PeTa18}, so we had to suitably modify the statements and adapt the proofs. 
The new difficulty is, as it was explained earlier, that the connectivity of the fibers of a bundle is not inherited by subbundles. 

\subsection{Slicing an exhausted bundle} \label{ssec:slicing}

\begin{lemma} \label{lem:subbundle-slicing}
  Let $\lb{M} = (\pi, M, B, F, f)$ be an exhausted bundle. 
  Let $A = \bigcup_{k = 1}^m A_k$, where $m \in \N$ and $A_k \in \Conn(f^{-1}(I_k))$ for some interval $I_k$ whose endpoints are regular values of $f$. 
  Then $A$ is a filled subspace of $\lb{M}$.
\end{lemma}

\begin{proof}
  Let $\Set{(U_i, \phi_i)}_{i \in \mathcal{I}}$ be a local trivialization.
  By continuity, $\phi_i(A_k \cap \pi^{-1}(U_i))=U_i \times P_{ik}$ for some $P_{ik}$, which is the disjoint union of some connected components of $h_i^{-1}(I_k)$.
  Hence if $P_i = \bigcup_{k = 1}^m P_{ik}$, then $\phi_i(A \cap \pi^{-1}(U_i)) = U_i \times P_i$.
  So $A$ is a filled subspace of $\lb{M}$.
\end{proof}

From \cref{lem:subbundle-slicing} we can conclude:

\begin{corollary} \label{cor:m-alpha}
  Let $\lb{M} = (\pi, M, B, F, f)$ be an exhausted bundle and $\alpha \in \mathbb{R}$.
  Let $x_0 \in M$ be a minimum\footnotemark of $f$.
  \begin{itemize}
    \item
      If $\alpha \in \mathbb{R} \setminus f(M)$ define $M_{[\alpha]} \defeq \emptyset$;
    \item
      If $\alpha \in \Regular(f) \cap f(M)$ and $C_\alpha$ be a connected component of $f^{-1}(-\infty, \alpha]$ containing $x_0$. 
      Define $M_{[\alpha]}$ to be the union of $C_\alpha$ and the precompact connected components of $M \setminus C_\alpha$.
  \end{itemize}   
  Then $M_{[\alpha]}$ is a compact and connected filled subspace of $\lb{M}$ with respect to any local trivialization.
\end{corollary}

\footnotetext{The bundles $\lb{M}_{[\alpha]}$ and $\lb{A}_{[\alpha]}$ will depend on the choice of $x_0$.
For this reason, we fix the choice of $x_0$ throughout the paper.}

By \cref{cor:m-alpha}, we can define
\begin{equation*}
  \lb{M}_{[\alpha]} \defeq \Res{\lb{M}}_{M_{[\alpha]}}
\end{equation*}
a filled subbundle of $\lb{M}$.
From \cref{lem:subbundle-slicing} we also have:

\begin{corollary} \label{cor:a-alpha}
Let $\lb{M} = (\pi, M, B, F, f)$ be an exhausted bundle and $\alpha \in \mathbb{R}$.
Let $A$ be a filled subspace of $\lb{M}$ with respect to $\Set{(U_i, \phi_i)}_{i \in \mathcal{I}}$. 
Let $P \subset F$ be the fiber of the filled subbundle $\lb{A} \defeq \Res{\lb{M}}_{A}$  of $\lb{M}$.
Then $A_{[\alpha]} \defeq A \cap M_{[\alpha]}$ is a filled subspace of $\lb{M}$ with respect to $\Set{(U_i, \phi_i)}_{i \in \mathcal{I}}$.
\end{corollary}

By \cref{cor:m-alpha}, we can define
\begin{equation*}
  \lb{A}_{[\alpha]} \defeq \Res{\lb{M}}_{A_{[\alpha]}}.
\end{equation*} 
a filled subbundle of $\lb{M}$.

The following gives the existence of $\lb{A}_{[\alpha]}$ with good properties: \emph{saturated slices}.
If $\lb{A}_{[\alpha]}$ is a saturated slice of $\lb{A}$ then the connected components of the fiber of $\lb{A}_{[\alpha]}$ are in one-to-one correspondence with those of $\lb{A}$.

\begin{lemma} \label{lem:saturating-threshold}
  Let $\lb{M} = (\pi, M, B, F, f)$ be an exhausted bundle with compact base. 
  Let $A$ be a connected filled subspace of $\lb{M}$ and $\lb{A} = \Res{\lb{M}}_A$.
  For any $\alpha \in \Regular(f) \cap f(A)$, let $\bfiota_\alpha \colon \lb{A}_{[\alpha]} \hookrightarrow \lb{A}$ be the embedding, and let $\Rls \bfiota_\alpha = (\iota_\alpha, \kappa_\alpha)$.
  Then map $\kappa_\alpha$ is a covering, see \cref{lem:covering-kappa}.
  If for any $\alpha' \in \Regular(f)$ with $\alpha' \geq \alpha$, $\kappa_\alpha$ is a diffeomorphism we call $\lb{A}_{[\alpha]}$ a saturated slice of $\lb{A}$ by $\alpha$.
  Then for any such $\lb{A}$ saturated slices of $\lb{A}$ exist.
\end{lemma}

For the proof of \cref{lem:saturating-threshold}, see \cref{sapp:three-technical}.

\subsection{A tree structure on a connected exhausted bundle}

The following generalizes \cite[Lemma 2.3]{PeTa18}. See \cref{sapp:trees} for a review of trees.

\begin{lemma} \label{lem:bundle-slicing-tree}
  Let $\lb{M} = (\pi, M, B, F, f)$ be a connected exhausted bundle, $\alpha_0 = -\infty$ and $\Set{\alpha_\ell}_{\ell \in \mathbb{N}} \subset \Regular(f) \cap f(M)$ be an unbounded strictly increasing sequence.
  Let $\mathcal{L}(\ell)$ be the collection of $\lb{A} = \Res{\lb{M}}_A$ where $A$ is any unbounded connected component of $M \setminus M_{[\alpha_{\ell-1}]}$. 
  Then there is a tree $(\mathcal{T}, \supsetneq)$ of filled subbundles of $\lb{M}$ such that $\mathcal{T} = \coprod_{\ell \in \N \cup \Set{0}} \mathcal{L}(\ell)$, where for $\lb{A}, \lb{C} \in \mathcal{T}$, $\lb{A} \supsetneq \lb{C}$ if $\lb{C}$ is a filled subbundle of (not equal to) $\lb{A}$.
  Moreover, $\Pa{\mathcal{T}, {\supsetneq}}$ is a rooted locally finite leafless tree of height $\bfomega$, and $\mathcal{L}(\ell) = \tlevel(\ell)$ for each $\ell \in \N \cup \Set{0}$.
\end{lemma}

\begin{proof}
  Let $\lb{A}_i \in \mathcal{L}(\ell_i) \subset \mathcal{T}$ where $\ell_i \in \mathbb{N} \cup \Set{0}$, for $i = 1, 2$ and $3$.
  By definition of connected components we have the following: if $\lb{A}_1 \supsetneq \lb{A}_2$, then $\ell_1 < \ell_2$; if $\lb{A}_1, \lb{A}_2 \supsetneq \lb{A}_3$ and $\ell_1 < \ell_2$, then $\lb{A}_1 \supsetneq \lb{A}_2$.
  Hence $\Pa{\mathcal{T}, \supsetneq}$ is a tree.
  The only root of $\mathcal{T}$ is $\lb{M} \in \tlevel(0)$, by induction $\mathcal{L}(\ell)$ is the $\ell$-th level of $\mathcal{T}$, which is finite, so $\mathcal{T}$ is locally finite.
  For any $\lb{A} \in \tlevel(\ell)$ with $A = \undsp(\lb{A})$, $A \setminus A_{[\alpha_{\ell+1}]} \neq \emptyset$, so $\mathcal{T}$ is leafless.
  Hence $\Setby{\tdepth(\lb{A})}{\lb{A} \in \mathcal{T}} = \mathbb{N} \cup \Set{0}$, and $\theight(\mathcal{T}) = \bfomega$.
\end{proof}

\subsection{Transferring volumes within fibers} \label{ssec:compact-support}

Next we will use the analytic tool, the work of Bueler--Prokhorenkov on Hodge theory \cite{MR1893604}, to prove a series of lemmas that allow us to move the volumes within the fibers of an exhausted bundle in various manners. 
First we explain how to move volumes within the interiors of compact submanifolds of the fibers. 

\begin{lemma} \label{lem:compact-exhausted-bundle}
  Let $\lb{M} = (\pi, M, B, F, f)$ be a filled bundle with compact base.
  Then following hold:
  \begin{enumerate}
    \item[\textup{(1)}]
      If $h_i$, $i \in \mathcal{I}$ in \cref{def:filled-bundle} are exhaustion functions, then $\lb{M}$ is an exhausted bundle.
    \item[\textup{(2)}]
      Suppose that $\lb{M}$ is also exhausted. 
      Then $B, F$ are compact if and only if $M$ is compact.
  \end{enumerate}
\end{lemma}

\begin{proof}
  Since $B$ is compact, let $\Set{(U_i, \phi_i)}_{i \in \mathcal{I}}$ be a local trivialization of $\lb{M}$ such that $\mathcal{I}$ is finite and $\overline{U_i}$ is compact for any $i \in \mathcal{I}$.
  Now that $h_i$ are exhaustion functions for $F$, let $\alpha \in \Regular(f)$, then
  \begin{equation*}
    f^{-1}((-\infty, \alpha]) = \bigcup_{i \in \mathcal{I}} f^{-1}((-\infty, \alpha]) \cap \pi^{-1}(\overline{U}_i) = \bigcup_{i \in \mathcal{I}} \phi_i \Pa{\overline{U}_i \times h_i^{-1}((-\infty, \alpha])}
  \end{equation*}
  is compact.
  Hence $f$ is an exhaustion function for $M$ and then $\lb{M}$ is an exhausted bundle. 
  This proves (1).

  If $\lb{M}$ is exhausted and $B$, $F$ are compact, then $M = \bigcup_{i \in \mathcal{I}} \pi^{-1}(\overline{U_i}) = \bigcup_{i \in \mathcal{I}} \phi_i \Pa{\overline{U_i} \times F}$ is compact.
  If $\lb{M}$ is exhausted and compact, then $B = \pi(M)$ is compact, and $f$ is bounded.
  So $h_i$ is bounded, which implies the compactness of $F$. 
  This proves (2).
\end{proof}

The following generalizes \cite[Lemma 3.1]{PeTa18}

\begin{lemma} \label{lem:fiber-collar-neighborhood}
  Let $\lb{M} = (\pi, M, B, F, f)$ be a filled bundle with compact base.
  Suppose that $N$ is a compact connected hypersurface of $M$ through regular points of $f$, on which $\pi$ is a submersion.
  Then:
  \begin{itemize}
    \item[{\rm (i)}]
      there exists $\varepsilon > 0$ and a diffeomorphism $\Phi \colon N \times (-\varepsilon, \varepsilon) \to V_N$ such that $V_N \subset M$ is a neighborhood of $N$, $\Phi(y, 0) = y$, $\pi(\Phi(y, s)) = \pi(y)$ and $f(\Phi(y, s)) = f(y) + s$ for any $(y, s) \in N \times (-\varepsilon, \varepsilon)$.
    \item[{\rm (ii)}]
      The set $V_N \setminus N$ has exactly two connected components, each characterized by the sign of $\proj_2 \circ \Phi^{-1}$.
    \item[{\rm (iii)}]
      If $N$ is a filled subspace of $\lb{M}$, $V_N$ is a filled subspace of $\lb{M}$.
    \item[{\rm (iiv)}]
      If $\lb{N} = \Res{\lb{M}}_N$ has connected fiber, $\lb{V}_{\lb{N}}= \Res{\lb{M}}_{V_N}$ has connected fiber.
  \end{itemize}
\end{lemma}

\begin{proof}
  Let $\mathrm{V}\lb{M} = \ker (\der \pi \colon \mathrm{T}M \to \mathrm{T}B)$ be the vertical tangent bundle of $\lb{M}$ and $g$ be any Riemannian metric on $M$.
  Let $Y \in \Gamma(\mathrm{V}\lb{M})$ be an extension of $\nabla\Pa{\Res{f}_{\pi^{-1}(p)}}$, the gradient of $\Res{f}_{\pi^{-1}(p)}$, for any $p \in B$.
  Then $\Res{Y(f)}_{\pi^{-1}(p)} = \Res{Y}_{\pi^{-1}(p)}\Pa{\Res{f}_{\pi^{-1}(p)}} = \abs{\nabla \Pa{\Res{f}_{\pi^{-1}(p)}}}_g^2$.
  Therefore there exists a neighborhood $\widetilde{V}_N \supset N$ such that $Y(f) > 0$ in $\widetilde{V}_N$.
  Let $X \in \Gamma(\mathrm{V}\lb{M})$ be such that $X(x) = \abs{\nabla \Pa{\Res{f}_{\pi^{-1}(\pi(x))}}(x)}_g^{-2} Y(x)$ for $x \in \widetilde{V}_N$, then $X(f) = 1$ on $\widetilde{V}_N$.
  Consider the flow of $X$, $\Phi \colon N \times (-\varepsilon, \varepsilon) \to M, (y, s) \mapsto x$, that is $\Phi(y, 0) = y$ for all $y \in N$ and $\frac{\partial \Phi}{\partial s}(y, s) = X(\Phi(y, s))$ for all $(y, s) \in N \times (-\varepsilon, \varepsilon)$, for $\varepsilon > 0$ small enough such that the image of $\Phi$ is contained in $\widetilde{V}_N$.
  Then we define $V_N = \Phi(N \times (-\varepsilon, \varepsilon))$.
  Since $X$ is vertical and $X(f) = 1$ in $V_N$, we have $\pi(\Phi(y, s)) = \pi(y)$ and $f(\Phi(y, s)) = f(y) + s$ for any $(y, s) \in N \times (-\varepsilon, \varepsilon)$, and $\Phi$ is a diffeomorphism.
  The two connected components of $V_N \setminus N$ are $\Phi(N \times (-\varepsilon, 0))$ and $\Phi(N \times (0, \varepsilon))$.
  If $N$ is a filled subspace of $\lb{M}$ then there is $\alpha \in \Regular(f)$ such that $N$ is the union of some connected components of $f^{-1}(\alpha)$, so $V_N$ is the union of some connected components of $f^{-1}((\alpha - \varepsilon, \alpha + \varepsilon))$ and by \cref{lem:subbundle-slicing} $V_N$ is a filled subspace of $\lb{M}$.
  If $\lb{N}$ has connected fiber, then since the fiber of $\lb{V}_{\lb{N}}$ is the image of the fiber of $\lb{N}$ under the flow of $X$ restricted on the fiber, it is connected.
\end{proof}

For any filled bundle $\lb{M}$ we define the \emph{fiber exterior derivative} 
\begin{equation*}
  \fibder \colon \Omega^q_F(\lb{M}) \rightarrow \Omega^{q+1}_F(\lb{M}), \eta \mapsto \fibder \eta
\end{equation*}
by $(\fibder \eta)_p = \der \eta_p$ for $p \in B$.

\begin{lemma} \label{thm:prelimitive-compact-support-bundle}
  Let $\lb{M} = (\pi, M, B, F, f)$ be a filled bundle with compact base.
  Suppose $W$ is an open subset of $M$ such that $\overline{W}$ is a compact submanifold with boundary $\partial W$, and $Z \subset F$ makes $\lb{W} = (\Res{\pi}_W, W, B, Z, \Res{f}_W)$ a filled subbundle of $\lb{M}$ with connected fiber.
  Let $\xi \in \pcspt{\Omega^k_F}(\lb{M})$.
  If $\support \xi \subset W$ and $\Res{\xi}_{W \cap \pi^{-1}(p)} \in \der \Omega^{k-1}_\cspt(W \cap \pi^{-1}(p))$ for any $p \in B$, then there is an $\eta \in \pcspt{\Omega^{k-1}_F}(\lb{M})$ such that $\xi = \fibder \eta$ and $\support \eta \subset \overline{W}$.
\end{lemma}

\begin{proof}
  Let $\Set{(U_i, \phi_i)}_{i \in \mathcal{I}}$ be a local trivialization of $\lb{M}$ with respect to which $W$ is a filled subspace of $\lb{M}$, then we can assume $\phi_i\Pa{W \cap \pi^{-1}(U_i)} = U_i \times Z$ for any $i \in \mathcal{I}$.
  By \cref{lem:compact-exhausted-bundle}, $\overline{Z}$ is compact.
  Let $\Set{\chi_i}_{i \in \mathcal{I}}$ be a partition of unity subordinated to the open cover $\Set{U_i}_{i \in \mathcal{I}}$ of $B$.
  We apply \cref{thm:prelimitive-compact-support} to $Z$ to get an operator $I_Z^q$, then define $\eta = \sum_{i \in \mathcal{I}} \phi_i^* I_Z^k \Pa{\phi_i^{-1}}^* ((\chi_i \circ \pi) \cdot \xi)$.
  Since $\der \circ I_Z^k = \identity$, we have $\xi = \fibder \eta$ and $\support \eta \subset \overline{W}$.
\end{proof}

The following is an extension of \cite[Lemma 4.5]{PeTa18}.

\begin{lemma} \label{lem:lemma1}
  Let $\lb{M} = (\pi, M, B, F, f)$ be a filled bundle with compact base.
  Let $\lb{V}$ be a filled subbundle of $\lb{M}$ with connected fiber.
  Suppose that $V = \undsp(\lb{V})$ is an open subset of $M$ such that $\overline{V}$ is a compact submanifold with boundary $\partial V$.
  Let $\omega, \tau \in \FVform(\lb{M})$ such that $\support(\omega - \tau) \subset V$ and 
  \begin{equation} \label{eq:lemma1-equal-volume}
    \int_{\lb{V}} \omega = \int_{\lb{V}} \tau.
  \end{equation}
  Then there is a fiber diffeomorphism $\varphi \colon M \to M$ such that $\varphi$ is the identity in a neighborhood of $M \setminus V$ and $\varphi^* \omega = \tau$.
\end{lemma}

\begin{proof}
  By \cref{lem:fiber-collar-neighborhood} applied to each connected component of $N = \partial V$ there exist $\varepsilon > 0$ and $V_N$ a neighborhood of $N$ satisfying (i)-(iii).
  Since $B$ is compact and $\support(\omega - \tau) \subset V$, we may reduce $\varepsilon$ as needed so that $\support(\omega - \tau) \subset V \setminus \overline{V_N}$.
  Let $W = V \setminus \overline{V_N}$ and  $W_p = W \cap \pi^{-1}(p)$, $p \in B$. It follows from \cref{eq:lemma1-equal-volume} that $\Res{(\omega_p - \tau_p)}_{W_p} \in \der \Omega_\cspt^{\dim F-1}(W_p)$.
  Therefore by \cref{thm:prelimitive-compact-support-bundle} there exists $\sigma \subset \pcspt{\Omega^{\dim F-1}_F}(\lb{M})$ with $\support \sigma \subset \overline{W}$ such that $\fibder \sigma = \omega - \tau$. 
  Define
  \begin{equation*}
    \omega_t = (1 - t)\omega + t\tau \in \FVform(\lb{M}) \quad \forall t \in [0, 1].
  \end{equation*}
  Since $\omega_t$ is nowhere vanishing there exists a smooth family of vertical vector fields $\Set{X_t}_{t \in [0, 1]} \subset \Gamma(\mathrm{V}\lb{M})$ where each $X_t$ is supported in $\overline{W}$ and such that 
  \begin{equation*}
    \omega_t(X_t, \cdot) = \sigma.
  \end{equation*}
  Let $\varphi_t \colon M \to M$ be a fiber diffeomorphism generated by $X_t$ that is the identity outside of $\overline{W}$.
  Then $\varphi = \varphi_1^{-1}$ satisfies the required properties.
\end{proof}

Now we carry out the transferring of volumes.
The following three lemmas correspond to \cite[Lemmas 4.6-4.8]{PeTa18} in the case of smooth families. 
The statements and the proofs are analogous but more delicate to implement due to the role that the release functor plays. 

\begin{lemma} \label{lem:volume-lemma}
  Let $\lb{M} = (\pi, M, B, F, f)$ be a filled bundle with compact base.
  Let $\lb{K}$ be a connected filled subbundle of $\lb{M}$ whose underlying space $K$ is a compact manifold with or without boundary which has a nonempty interior. 
  Let $B_K$ be the base of $\Rls \lb{K}$, and let $w \in \smth(B_K; \mathbb{R})$.
  If $\omega \in \FVform(\lb{M})$, then there exists $\tau \in \FVform(\lb{M})$ such that $\support(\omega - \tau) \subset K^\circ$ and 
  \begin{equation*}
    \rint_{\lb{K}} \tau = w.
  \end{equation*}  
\end{lemma}

\begin{proof}
  Let $\xi \in \FVform(\lb{M})$ be such that $\support(\xi - \omega) \subset K^\circ$ and $\rint_{\lb{K}} \xi < w$.
  Let $\eta \geq 0$ be a fiber top-form on $\lb{M}$ such that $\support \eta \subset K^\circ$ and $\rint_{\lb{K}} \eta > 0$.
  Define
  \begin{equation*}
    \tau = \xi + \Rls (\Res{\pi}_K)^*\Pa{\frac{w - \rint_{\lb{K}} \xi}{\rint_{\lb{K}} \eta}} \eta,
  \end{equation*}
   where $\Rls (\Res{\pi}_K)$ be the bundle map of $\Rls(\lb{K})$.
\end{proof}

\begin{lemma} \label{lem:lemma2}
  Let $\lb{W} = (\pi, W, B, Z, f)$ be a filled bundle with compact base and oriented connected fiber.
  Let $\lb{N}$ be a filled subbundle of $\lb{W}$ with connected fiber. 
  Suppose that $N = \undsp(\lb{N})$ is a connected hypersurface of $W$ such that $W \setminus N$ has two components, say $W^+$ and $W^-$.
  Let $\lb{W}^+ = \Res{\lb{W}}_{W^+}$ and $\lb{W}^- = \Res{\lb{W}}_{W^-}$.
  Let $\omega, \tau \in \FVform(\lb{W})$. 
  Then there is a neighborhood $V_N$ of $N$, and a fiber diffeomorphism $\varphi \colon W \to W$ with the following properties:
  \begin{itemize}
    \item[\textup{(1)}]
      $\varphi$ is the identity in a neighborhood of $W \setminus V_N$;
    \item[\textup{(2)}]
      $\varphi^*\omega = \tau$ in a neighborhood of $N$;
    \item[\textup{(3)}]
      $\int_{\lb{W}^+} \varphi^* \omega = \int_{\lb{W}^+} \omega$; and $\int_{\lb{W}^-} \varphi^* \omega = \int_{\lb{W}^-} \omega$.
  \end{itemize}
\end{lemma}

\begin{proof}
  By \cref{lem:fiber-collar-neighborhood}, there exists $\lb{V}_{\lb{N}}$ with underlying space $V_N$ as a filled subbundle of $\lb{W}$, $\varepsilon > 0$ and a diffeomorphism $\Phi \colon N \times (-\varepsilon, \varepsilon) \to V_N$ such that $V_N \subset M$ is a neighborhood of $N$, $\Phi(y, 0) = y$, $\pi(\Phi(y, s)) = \pi(y)$ and $f(\Phi(y, s)) = f(y) + s$ for any $(y, s) \in N \times (-\varepsilon, \varepsilon)$.
  Let $V_N^+ = \Phi(N \times (0, \varepsilon))$ and $V_N^- = \Phi(N \times (-\varepsilon, 0))$.
  We consider $\Phi(N \times [0, \varepsilon))$ first. 
  By compactness of $B$ there exists $0 < \delta < \varepsilon / 2$ such that 
  \begin{equation*}
    \int_{\Res{\lb{M}}_{\Phi(N \times (0, \varepsilon-\delta))}} \tau > \int_{\Res{\lb{M}}_{\Phi(N \times (0, \delta))}} \omega \quad \text{and} \quad \int_{\Res{\lb{M}}_{\Phi(N \times (0, \varepsilon-\delta))}} \omega > \int_{\Res{\lb{M}}_{\Phi(N \times (0, \delta))}} \tau.
  \end{equation*}

  Let $s = \proj_2 \circ \Phi^{-1} \colon \Phi(N \times (-\varepsilon, \varepsilon)) \to (-\varepsilon, \varepsilon)$ and choose a smooth function $\zeta \colon (0, \varepsilon) \times (0, 1) \to [0,1]$ such that $\zeta(s, \cdot) = 1$ if $s \in (0, \delta]$, $\zeta(s, \cdot) = 0$ if $s \in [\varepsilon - \delta, \varepsilon)$, $\lim_{t \to 0+} \zeta(s, t) = 0$, $\frac{\partial \zeta}{\partial t}(s, \cdot) > 0$, and $\lim_{t \to 1-} \zeta(s, t) = 1$ if $s \in (\delta, \varepsilon - \delta)$.
  Let $\lb{V}_{\lb{N}}^+ = \Res{\lb{M}}_{V_N^+}$ and $\lb{V}_{\lb{N}}^- = \Res{\lb{M}}_{V_N^-}$.
  Then $\theta(t, \cdot) = \int_{\lb{V}_{\lb{N}}^+} \zeta(s(\cdot), t) \tau - \int_{\lb{V}_{\lb{N}}^+} \zeta(s(\cdot), 1 - t) \omega$ is smooth on $(0, 1) \times B$ and $\frac{\partial \theta}{\partial t}(t, \cdot) = \int_{\lb{V}_{\lb{N}}^+} \frac{\partial \zeta}{\partial t}(s(\cdot), t) \tau + \int_{\lb{V}_{\lb{N}}^+} \frac{\partial \zeta}{\partial t}(s(\cdot), 1 - t) \omega > 0$ for any $t \in (0, 1)$ and $\lim_{t \to 0+} \theta(t, p) < 0 < \lim_{t \to 1-} \theta(t, p)$ for any $p \in B$.
  Then for every $p \in B$ there is a unique smooth $t = t(p) \colon B \to \mathbb{R}$ solving $\theta(t(p), p) = 0$.
  The functions $\lambda(x) \defeq \zeta(s(x), t(\pi(x)))$ and $\mu(x) \defeq \zeta(s(x), 1 - t(\pi(x)))$ on $V_N^+$ are smooth in $x$.
  By analogy we define $\lambda$ and $\mu$ in $V_N^-$, then let $\lambda = \mu = 1$ on $N$ and $\lambda = \mu = 0$ in $W \setminus V_N$.
  In this way, $\lambda, \mu \colon W \to \mathbb{R}$ are defined smoothly, such that $\int_{\lb{V}_{\lb{N}}^+} \mu \omega = \int_{\lb{V}_{\lb{N}}^+} \lambda \tau$ and $\int_{\lb{V}_{\lb{N}}^-} \mu \omega = \int_{\lb{V}_{\lb{N}}^-} \lambda \tau$.

  Hence $\int_{\lb{V}_{\lb{N}}^+} \Pa{(1 - \mu) \omega + \lambda \tau} = \int_{\lb{V}_{\lb{N}}^+} \omega$ and $\int_{\lb{V}_{\lb{N}}^-} \Pa{(1 - \mu) \omega + \lambda \tau} = \int_{\lb{V}_{\lb{N}}^-} \omega$. 
  It follows from \cref{lem:lemma1} applied to $(1 - \mu) \omega + \lambda \tau$ and $\omega$ on $\lb{V}_{\lb{N}}^+$ and $\lb{V}_{\lb{N}}^-$ that there exists a fiber diffeomorphism $\varphi \colon W \to W$ such that $\varphi = \identity$ in $V \setminus \Phi(N \times (\delta - \varepsilon, \varepsilon - \delta))$ and $\varphi^* \omega = (1 - \mu) \omega + \lambda \tau$. 
  Hence $\varphi$ satisfies the conditions claimed in the statement.
\end{proof}

\begin{lemma} \label{lem:lemma3}
  Let $\lb{M} = (\pi, M, B, F, f)$ be a connected exhausted bundle with compact base and oriented noncompact connected fibers.
  Let $\Set{L_j}_{j \in \mathbb{N}}$ be a cover of $M$ by compact connected submanifolds with boundary, which have the same dimension as $M$, and whose interiors are pairwise disjoint.
  Suppose for any $j \in \mathbb{N}$, $L_j$ is a filled subspace.
  Let $\lb{L}_j = \Res{\lb{M}}_{L_j}$.
  If $\omega, \tau \in \FVform(\lb{M})$ are such that $\rint_{\lb{L}_j} \omega = \rint_{\lb{L}_j} \tau$ for each $j \in \mathbb{N}$ then there is a fiber diffeomorphism $\varphi \colon M \to M$ such that $\varphi^* \omega = \tau$.
\end{lemma}

\begin{proof}
  By the construction of $\Set{\lb{L}_j}_{j \in \mathbb{N}}$, any three different $L_j$'s for $j \in \mathbb{N}$ do not intersect.
  Let
  \begin{equation*}
    \mathcal{C} = \Setby{\Res{\lb{M}}_N}{N \in \Conn{L_j \cap L_k}, j, k \in \mathbb{N}, j \neq k}.
  \end{equation*}  
  Then $\mathcal{C}$ is a collection of pairwise disjoint filled subbundles of $\lb{M}$ whose underlying spaces are hypersurfaces of $M$.
  So for each $\lb{N} \in \mathcal{C}$ with underlying space $N$, let $j, k \in \mathbb{N}$ be such that $N \subset L_j \cap L_k$, by \cref{lem:fiber-collar-neighborhood}, we obtain $\varepsilon_N > 0$ and a diffeomorphism $\Phi_N \colon V_N \times (-\varepsilon_N, \varepsilon_N) \to V_N$ where $V_N$ is a neighborhood of $N$ and a filled subspace of $\lb{M}$. 
  We require $V_N^- \subset L_j$ and $V_N^+ \subset L_k$.
  Let $\lb{V}_{\lb{N}} = \Res{\lb{M}}_{V_N}$, $\lb{V}_{\lb{N}}^- = \Res{\lb{M}}_{V_N^-}$, and $\lb{V}_{\lb{N}}^+ = \Res{\lb{M}}_{V_N^+}$.
  Now apply \cref{lem:lemma2} to $\Rls \lb{V}_{\lb{N}}$ in order to obtain a fiber diffeomorphism $\varphi_N\colon M \to M$ such that $\varphi_N = \identity$ in a neighborhood of $M \setminus V_N$, $\varphi_N^*\omega = \tau$ in a neighborhood of $N$, and $\rint_{\lb{V}_{\lb{N}}^+} \varphi_N^* \omega = \rint_{\lb{V}_{\lb{N}}^+} \omega$ as well as $\rint_{\lb{V}_{\lb{N}}^-} \varphi_N^* \omega = \rint_{\lb{V}_{\lb{N}}^-} \omega$ (note that a differmorphism preserving the released bundle map also preserves the original bundle map). 
  Therefore $\rint_{\lb{L}_j} \varphi_N^* \omega = \rint_{\lb{L}_j} \omega, \rint_{\lb{L}_k} \varphi_N^* \omega = \rint_{\lb{L}_k} \omega$.
  Note that $V_N^- \subset L_j$, so the base of $\Rls \lb{V}_{\lb{N}}^-$ covers that of $\Rls \lb{L}_j$.

  If necessary, choose $\varepsilon_N$ small such that the family $\Set{\overline{V_N}}_{N \in \mathcal{C}}$ is mutually disjoint.
  Since replacing $\omega$ by $\varphi_N^*\omega$ each time does not change the released fiber volume of $\lb{L}_j$ for any $j \in \mathbb{N}$, we compose these $\varphi_N$ for $N = \undsp(\lb{N})$, for $\lb{N} \in \mathcal{C}$, as they are the identity away from disjoint open sets, to obtain a fiber diffeomorphism $\varphi' \colon M \to M$ such that $\omega' = \varphi'^* \omega$ is equal to $\tau$ on some neighborhood of $\bigcup_{N \in \mathcal{C}} N$ and $\rint_{\lb{L}_j} \omega' = \rint_{\lb{L}_j} \omega = \rint_{\lb{L}_j} \tau$ for each $j \in \mathbb{N}$.
  Applying \cref{lem:lemma1} to each $\Rls \lb{L}_j$ for $j \in \mathbb{N}$ we get a fiber diffeomorphism $\psi_j \colon M \to M$ such that $\tau = \psi_j^* \omega'$ in $L_j$ and $\psi_j = \identity$ in a neighborhood of $M \setminus L_j$.
  Replacing $\omega'$ by $\psi_j^* \omega'$ each time and composing $\Set{\psi_j}_{j \in \mathbb{N}}$ we obtain a fiber diffeomorphism $\psi' \colon M \to M$ such that $\tau = \psi'^* \omega'$.
  Then $\varphi = \varphi' \circ \psi'$ satisfies the required properties.
\end{proof}

\section{Proof of the main result} \label{sec:filtration-theorem}

The analogue of the following lemma for the case of smooth families appeared in \cite[Lemma 4.1]{PeTa18}, the proof strategy is analogous.

For any tree $\mathcal{T}$ of height $\bfomega$, for any $\lb{X} \in \tlevel(\ell)$, $\ell \in \N \cup \Set{0}$, let
\begin{equation*} 
  \Tse_{\mathcal{T}} \lb{X} = \lb{X}_{[\alpha_{\ell+1}]}, \qquad \Shta_{\mathcal{T}} \lb{X} = \lb{X}_{[\alpha_{\ell+2}]}.
\end{equation*}
Then we have
\begin{equation*}
  \undsp\Pa{\Tse_{\mathcal{T}} \lb{X}} = \undsp(\lb{X}) \setminus \coprod_{\lb{Y} \in \tsuc(\lb{X})} \undsp(\lb{Y}), \quad \undsp\Pa{\Shta_{\mathcal{T}} \lb{X}} = \undsp(\lb{X}) \setminus \coprod_{\lb{Z} \in \tssuc(\lb{X})} \undsp(\lb{Z}).
\end{equation*}

\begin{lemma} \label{lem:filtration}
  Let $\lb{M} = (\pi, M, B, F, f)$ be a connected exhausted bundle with compact base and oriented noncompact connected fiber.
  Let $\omega, \tau \in \FVform(\lb{M})$ be commensurable and suppose that they have equal fiber integral.
  Then there is a tree $\Pa{\mathcal{T}, {\supsetneq}}$ of connected filled subbundles of $\lb{M}$ and $\Set{\omega_n}_{n \in \mathbb{N} \cup \Set{0}}, \Set{\tau_n}_{n \in \mathbb{N} \cup \Set{0}} \subset \FVform(\lb{M})$ such that $\omega_0 = \omega, \tau_0 = \tau$ and for any $n \in \mathbb{N}$, we have that
  \begin{equation} \label{eq:support-in-level}
    \support(\omega_n - \omega_{n-1}) \cup \support(\tau_n - \tau_{n-1}) \subset \bigcup_{\lb{C} \in \tlevel(2n-2)} \Pa{\undsp(\Shta_\mathcal{T} \lb{C})}^\circ,
  \end{equation}
  as well as that for each $\lb{A} \in \tlevel(2n-3)$ with $n > 1$, $\lb{C} \in \tlevel(2n-2)$, $\lb{E} \in \tlevel(2n-1)$,
  \begin{align}
    \rint_{\Tse_\mathcal{T} \lb{M}} \omega_1 &= \rint_{\Tse_\mathcal{T} \lb{M}} \tau_1, &\rint_{\Shta_\mathcal{T} \lb{A}} \omega_n &= \rint_{\Shta_\mathcal{T} \lb{A}} \tau_n \text{~for $n > 1$}; \label{eq:volume-in-Shta-A} \\
    \rint_{\Shta_\mathcal{T} \lb{C}} \omega_n &= \rint_{\Shta_\mathcal{T} \lb{C}} \omega_{n-1},  &\rint_{\Shta_\mathcal{T} \lb{C}} \tau_n &= \rint_{\Shta_\mathcal{T} \lb{C}} \tau_{n-1}; \label{eq:volume-in-Shta-C} \\
    \rint_{\lb{E}} \omega_n &= \rint_{\lb{E}} \tau_n. & & \label{eq:volume-in-E}
  \end{align}
\end{lemma}

\inputfigure{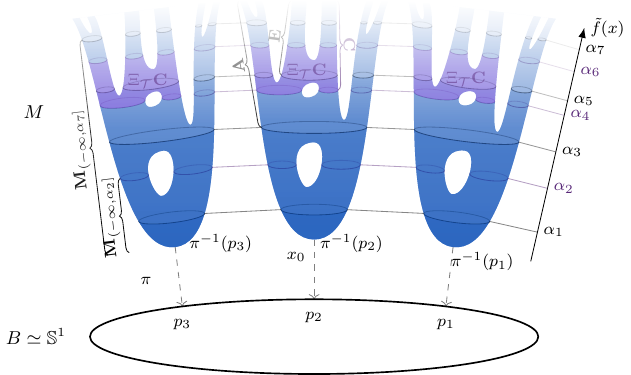}{preparation}{
$\lb{M} = (\pi, M, B, F, f)$ in \cref{lem:filtration}.
}

\begin{proof}
  Our goal is to find $\alpha_0 = -\infty$ and $\Set{\alpha_\ell}_{\ell \in \N} \subset \Regular(f) \cap f(M)$ such that $\mathcal{T}$ is constructed by \cref{lem:bundle-slicing-tree}, see \cref{fig:preparation}.
  Note that, if we know $\Set{\alpha_\ell}_{0 \leq \ell \leq m}$ for some $m \in \N \cup \Set{0}$ for the sequence $\Set{\alpha_\ell}_{\ell \in \N \cup \Set{0}}$ defining $\mathcal{T}$, then we say $\mathcal{T}$ is constructed up to the $m$-th level, so we know $\tlevel(\ell)$ of $\mathcal{T}$ for any $\ell$ with $0 \leq \ell \leq m$.
  We proceed by induction on $n \in \mathbb{N} \cup \Set{0}$ to find $\alpha_{2n-1}$, $\alpha_{2n}$ and $\omega_n, \tau_n \in \FVform(\lb{M})$ such that $\rint_{\lb{E}} \omega_n = \rint_{\lb{E}} \tau_n$ for any $\lb{E} \in \tlevel(2n-1)$ ($\lb{E} \in \tlevel(0)$ if $n = 0$).
  For any $\lb{X}, \lb{Y} \in \mathcal{T}$ with $\lb{X} \supsetneq \lb{Y}$, let $\bfiota_{\lb{X}}^{\lb{Y}} \colon \lb{Y} \hookrightarrow \lb{X}$ be the embedding and let 
  \begin{equation*}
   \Rls \bfiota_{\lb{X}}^{\lb{Y}} = (\iota_{\lb{X}}^{\lb{Y}}, \kappa_{\lb{X}}^{\lb{Y}}).
  \end{equation*} 
  By \cref{lem:covering-kappa}, $\kappa_{\lb{X}}^{\lb{Y}}$ is a covering map, whose number of sheets is denoted by $\#\kappa_{\lb{X}}^{\lb{Y}}$.

  \emph{Case $0$.}
  Set $\alpha_0 = -\infty$, then $\lb{M}_{[\alpha_0]} = \emptyset$, and $\tlevel(0) = \Set{\lb{M}}$.
  Since $\omega$ and $\tau$ has equal fiber integral and $\lb{M}$ has connected fiber, we have
  $\rint_{\lb{M}} \omega_0 = \rint_{\lb{M}} \tau_0$.

  \medskip
  
  \emph{Case $(n-1)$ for $n \in \mathbb{N}$.}
  By induction we assume that for any $\lb{A} \in \tlevel(2n-3)$ ($\lb{A} \in \tlevel(0)$ when $n = 1$) we have:
  \begin{equation} \label{eq:equal-volume-A}
    \rint_{\lb{A}} \omega_{n-1} = \rint_{\lb{A}} \tau_{n-1}.
  \end{equation}
  
  \medskip

  \emph{Case $n$ for $n \in \mathbb{N}$.}
  Take $\alpha_{2n-1} \in \Regular(f)$ such that $\lb{C}_{[\alpha_{2n-1}]}$ is a saturated slice of $\lb{C}$ for each $\lb{C} \in \tlevel(2n-2)$ (for the concept of saturated slices and the existence of $\alpha_{2n-1}$, see \cref{lem:saturating-threshold}).
  Then $\mathcal{T}$ is constructed up to the $(2n-1)$-th level.
  Let $\lb{A} \in \tlevel(2n-3)$ (if $n = 1$ let $\lb{A} = \lb{M}$ and replace $\tssuc(\lb{A})$ by $\tsuc(\lb{M})$, $\Shta_\mathcal{T} \lb{A}$ by $\Tse_\mathcal{T} \lb{M}$ throughout this paragraph).
  The base of $\Rls \lb{A}$ is $B_A$.
  Let $\tssuc_0(\lb{A})$ (resp. $\tssuc_1(\lb{A})$) be the subcollection of elements in $\tssuc(\lb{A})$ with finite (resp. infinite) volume.
  For any $\lb{E} \in \tssuc(\lb{A})$, we define $\delta_{\lb{E}} \in \smth(B_E; \mathbb{R})$ as follows: if $\lb{E}$ has finite volume, let $\delta_{\lb{E}} = \rint_{\lb{E}} \omega_{n-1} - \rint_{\lb{E}} \tau_{n-1}$; if $\lb{E}$ has infinite volume, let
  \begin{equation*}
    \delta_{\lb{E}} = \dfrac{1}{\sum\limits_{\lb{G} \in \tssuc_1(\lb{A})} \# \kappa^{\lb{G}}_{\lb{A}}} (\kappa^{\lb{E}}_{\lb{A}})^* \Pa{ \rint_{\Shta_\mathcal{T} \lb{A}} \tau_{n-1} - \rint_{\Shta_\mathcal{T} \lb{A}} \omega_{n-1} - \sum_{\lb{G} \in \tssuc_0(\lb{A})} (\kappa^{\lb{G}}_{\lb{A}})_* \delta_{\lb{G}}}.
  \end{equation*}
  Then combining equations \cref{eq:covering-space-equation} and \cref{eq:equal-volume-A} we obtain
  \begin{equation*}
    \sum_{\lb{E} \in \tssuc(\lb{A})} (\kappa^{\lb{E}}_{\lb{A}})_* \delta_{\lb{E}} = \rint_{\Shta_\mathcal{T} \lb{A}} \tau_{n-1} - \rint_{\Shta_\mathcal{T} \lb{A}} \omega_{n-1}.
  \end{equation*}

  For every $\lb{C} \in \tsuc(\lb{A})$, let $B_C$ be the base of $\Rls \lb{C}$, and let $u_{\lb{C}} \in \smth(B_C; \mathbb{R})$ be such that
  \begin{align*}
    \max\Pa{- \rint_{\Tse_\mathcal{T} \lb{C}} \omega_{n-1}, - \rint_{\Tse_\mathcal{T} \lb{C}} \tau_{n-1} + \sum_{\lb{E} \in \tsuc(\lb{C})} (\kappa^{\lb{E}}_{\lb{C}})_* \delta_{\lb{E}}} < u_{\lb{C}} < \rint_{\lb{C}} \omega_{n-1} - \rint_{\Tse_\mathcal{T} \lb{C}} \omega_{n-1}. 
  \end{align*}
  If $\lb{C}$ has finite volume then
  \begin{align*}
    &\phantom{{}=} \Pa{\rint_{\lb{C}} \omega_{n-1} - \rint_{\Tse_\mathcal{T} \lb{C}} \omega_{n-1}} - \Pa{ -\rint_{\Tse_\mathcal{T} \lb{C}} \tau_{n-1} + \sum_{\lb{E} \in \tsuc(\lb{C})} (\kappa^{\lb{E}}_{\lb{C}})_* \delta_{\lb{E}}} \\
    &= \rint_{\lb{C}} \omega_{n-1} + \Pa{\rint_{\Tse_\mathcal{T} \lb{C}} \tau_{n-1} - \rint_{\Tse_\mathcal{T} \lb{C}} \omega_{n-1}} + \sum_{\lb{E} \in \tsuc(\lb{C})} (\kappa^{\lb{E}}_{\lb{C}})_* \Pa{\rint_{\lb{E}} \tau_{n-1} - \rint_{\lb{E}}  \omega_{n-1}} \\
    &= \rint_{\lb{C}} \tau_{n-1} > 0,
  \end{align*}
  so such $u_{\lb{C}}$ exists.
  Since 
  \begin{equation*}
    u_{\lb{C}} < \sum_{\lb{E} \in \tsuc(\lb{C})} (\kappa^{\lb{E}}_{\lb{C}})_* \rint_{\lb{E}} \omega_{n-1} = \rint_{\lb{C}} \omega_{n-1} - \rint_{\Tse_\mathcal{T} \lb{C}} \omega_{n-1},
  \end{equation*}
  by \cref{lem:approximation-lemma} applied to the covering map
  \begin{equation*}
    \coprod_{\lb{E} \in \tsuc(\lb{C})} B_E \to B_C,
  \end{equation*}
  we may choose $v_{\lb{E}} \in \smth(B_E; \mathbb{R})$ (where $B_E$ is the base of $\Rls \lb{E}$) such that $v_{\lb{E}} < \rint_{\lb{E}} \omega_{n-1}$ and $\sum_{\lb{E} \in \tsuc(\lb{C})} (\kappa^{\lb{E}}_{\lb{C}})_* v_{\lb{E}} = u_{\lb{C}}$.

  For any $\lb{E} \in \tsuc(\lb{C})$, if $\lb{E}$ has infinite volume, take $\beta_{\lb{E}} \in \Regular(f)$ such that $\lb{E}_{[\beta_{\lb{E}}]}$ is a saturated slice (for the concept of saturated slices and the existence of $\alpha_{2n-1}$, see \cref{lem:saturating-threshold}).
  Otherwise, the function 
  \begin{equation} \label{eq:lambdaps}
    (\cdot, \beta) \mapsto \min\Pa{\int_{\Res{(\Rls \lb{E})}_{E_{[\beta]}}} \omega_{n-1}, \int_{\Res{(\Rls \lb{E})}_{E_{[\beta]}}} \tau_{n-1} + \delta_{\lb{E}}} - v_{\lb{E}}
  \end{equation}
  $B_E \times \mathbb{R}$ is continuous in the first variable, is increasing in $\beta$, and converges to $\rint_{\lb{E}} \omega_{n-1} - v_{\lb{E}} > 0$ as $\beta \to +\infty$ pointwise.
  Note $\Rls \lb{E}$ is not a subbundle of $\lb{M}$, so we cannot slice it by $\beta$.
  Since $B_E$ is compact there is $\beta_{\lb{E}} > \alpha_{2n-1}$ such that \cref{eq:lambdaps} is positive when $\beta = \beta_{\lb{E}}$.
  Let $\alpha_{2n} = \max_{\lb{E} \in \tlevel(2n-1)} \beta_E$, then $\mathcal{T}$ is constructed up to the $2n$-th level.
  So $\Tse_\mathcal{T} \lb{E} = \lb{E}_{[\alpha_{2n}]}$, then we have $v_{\lb{E}} < \rint_{\Tse_\mathcal{T} \lb{E}} \omega_{n-1}$, and $v_{\lb{E}} - \delta_{\lb{E}} < \rint_{\Tse_\mathcal{T} \lb{E}} \tau_{n-1}$.  Since all the right hand sides of these expressions are positive, by \cref{lem:volume-lemma}, there are $\omega_n, \tau_n \in \FVform(\lb{M})$ such that
  \begin{align*}
    \rint_{\Tse_\mathcal{T} \lb{C}} \omega_n &= \rint_{\Tse_\mathcal{T} \lb{C}} \omega_{n-1} + u_{\lb{C}}, &\rint_{\Tse_\mathcal{T} \lb{C}} \tau_n &= \rint_{\Tse_\mathcal{T} \lb{C}} \tau_{n-1} + u_{\lb{C}} - \sum_{\lb{E} \in \tsuc(\lb{C})} (\kappa^{\lb{E}}_{\lb{C}})_* \delta_{\lb{E}}, \\
    \rint_{\Tse_\mathcal{T} \lb{E}} \omega_n &= \rint_{\Tse_\mathcal{T} \lb{E}} \omega_{n-1} - v_{\lb{E}} , &\rint_{\Tse_\mathcal{T} \lb{E}} \tau_n &= \rint_{\Tse_\mathcal{T} \lb{E}} \omega_{n-1} - (v_{\lb{E}} - \delta_{\lb{E}}),
  \end{align*}
  and
  \begin{multline*}
    \support(\omega_n - \omega_{n-1}) \cup \support(\tau_n - \tau_{n-1}) \\
    \subset (M_{[\alpha_{2n}]})^\circ \setminus M_{[\alpha_{2n-2}]} = \bigcup_{\lb{C} \in \tlevel(2n-2)} \Pa{\undsp(\Shta_\mathcal{T} \lb{C})}^\circ.
  \end{multline*}

  Then we have
  \begin{align*}
    \rint_{\Shta_\mathcal{T} \lb{A}} \omega_n &= \rint_{\Shta_\mathcal{T} \lb{A}} \omega_{n-1} + \sum_{\lb{C} \in \tsuc(\lb{A})} (\kappa^{\lb{C}}_{\lb{A}})_* u_{\lb{C}} \\
     &= \rint_{\Shta_\mathcal{T} \lb{A}} \tau_{n-1} - \sum_{\lb{E} \in \tssuc(\lb{A})} (\kappa^{\lb{C}}_{\lb{A}})_* \Pa{(\kappa^{\lb{E}}_{\lb{C}})_* \delta_{\lb{E}} - u_{\lb{C}}} = \rint_{\Shta_\mathcal{T} \lb{A}} \tau_n,
  \end{align*}
  and
  \begin{align*}
    \rint_{\Shta_\mathcal{T} \lb{C}} \omega_n &= \rint_{\Tse_\mathcal{T} \lb{C}} \omega_n + \sum_{\lb{E} \in \tsuc(\lb{C})} (\kappa^{\lb{E}}_{\lb{C}})_* \rint_{\Tse_\mathcal{T} \lb{E}} \omega_n = \rint_{\Shta_\mathcal{T} \lb{C}} \omega_{n-1}, \\
    \rint_{\Shta_\mathcal{T} \lb{C}} \tau_n &= \rint_{\Tse_\mathcal{T} \lb{C}} \tau_n + \sum_{\lb{E} \in \tsuc(\lb{C})} (\kappa^{\lb{E}}_{\lb{C}})_* \rint_{\Tse_\mathcal{T} \lb{E}} \tau_n = \rint_{\Shta_\mathcal{T} \lb{C}} \tau_{n-1},
  \end{align*}
  and
  \begin{equation*}
  \begin{split}
    \rint_{\lb{E}} \omega_n &= \rint_{\Tse_\mathcal{T} \lb{E}} \omega_n + \rint_{\lb{E}} \omega_{n-1} - \rint_{\Tse_\mathcal{T} \lb{E}} \omega_{n-1} \\
    &= \rint_{\lb{E}} \omega_{n-1} - v_{\lb{E}} = \rint_{\lb{E}} \tau_{n-1} - (v_{\lb{E}} - \delta_{\lb{E}}) = \rint_{\lb{E}} \tau_n.
  \end{split}
  \end{equation*}
\end{proof}

Now we can apply \cref{lem:filtration} to $\lb{M}$ and $\omega, \tau$, in which way we obtain the tree $\mathcal{T}$ of filled subbundles of $\lb{M}$ such that \cref{eq:support-in-level,eq:volume-in-Shta-A,eq:volume-in-Shta-C,eq:volume-in-E} hold.

For $n \in \mathbb{N}$ and $\lb{C} \in \tlevel(2n-2)$, applying \cref{lem:lemma1} to $\Pa{\Shta_\mathcal{T} \lb{C}}^\circ$, there are fiber diffeomorphisms $\varphi_n, \psi_n \colon M \to M$ such that $\varphi_n ^* \omega_{n-1} = \omega_n$, $\psi_n ^* \tau_{n-1} = \tau_n$, and $\varphi_n = \psi_n = \identity$ outside of $(M_{[\alpha_{2n}]})^\circ \setminus M_{[\alpha_{2n-2}]}$.
Let
\begin{equation} \label{eq:omega-infinity}
\begin{aligned}
  \omega_\infty &= \lim_{n \to \infty} \omega_n, & \tau_\infty &= \lim_{n \to \infty} \tau_n, \\
  \varphi_\infty &= \varphi_1 \circ \varphi_2 \circ \cdots, & \psi_\infty &= \psi_1 \circ \psi_2 \circ \cdots.
\end{aligned}
\end{equation}
Since the interiors of $\undsp(\Shta_\mathcal{T} \lb{C})$, $\lb{C} \in \mathcal{T}$ with even depths are mutually disjoint, the pointwise limits in \cref{eq:omega-infinity} will be stable at a finite $n$, so $\omega_\infty, \tau_\infty \in \FVform(\lb{M})$, $\varphi_\infty, \psi_\infty \colon M \to M$ must be fiber diffeomorphisms, 
\begin{equation*}
  \rint_{\Tse_\mathcal{T} \lb{M}} \omega_\infty = \rint_{\Tse_\mathcal{T} \lb{M}} \tau_\infty, \qquad \rint_{\Shta_\mathcal{T} \lb{A}} \omega_\infty = \rint_{\Shta_\mathcal{T} \lb{A}} \tau_\infty
\end{equation*}
for each $\lb{A} \in \mathcal{T}$ with odd depth, $\varphi_\infty^* \omega = \omega_\infty$, and $\psi_\infty^* \tau = \tau_\infty$.

Let $\Set{\lb{L}_j}_{j \in \mathbb{N}}$ be the set of $\Tse_\mathcal{T} \lb{M}$ and the closures of $\Shta_\mathcal{T} \lb{A}$ for $\lb{A} \in \mathcal{T}$ with even depths.
By \cref{lem:lemma3}, there is a fiber diffeomorphism $\varphi' \colon M \to M$ such that $\varphi'^* \omega_\infty = \tau_\infty$.

Finally, $\varphi = \varphi_\infty \circ \varphi' \circ \psi_\infty^{-1} \colon M \to M$, which concludes the proof.

\appendix

\section{Technical tools} \label{app:technical-tools}

\subsection{Trees} \label{sapp:trees}

A \emph{tree} is a strictly partially ordered set $(\mathcal{T}, \prec)$ with the property that for each $x \in \mathcal{T}$, the set $\tPre(x) = \Setby{y \in \mathcal{T}}{y \prec x}$ of all predecessors of $x$ is well ordered by $\prec$.
We write $\mathcal{T}$ for $(\mathcal{T}, \prec)$ when there is no ambiguity.
Let $\tRoot(\mathcal{T}) = \Setby{x \in T}{\forall y \in T, y \not\prec x} \neq \emptyset$ be the set of roots of $\mathcal{T}$.
If $\tRoot(\mathcal{T})$ is a singleton we call $\mathcal{T}$ \emph{rooted}.

Let $\tSuc(x) = \Setby{y \in \mathcal{T}}{y \succ x}$ be the set of all successors of $x$, then $(\tSuc(x), \prec)$ is a tree.
Let $\tsuc(x) = \tRoot(\tSuc(x))$ be the set of \emph{children} of $x$.
If for any $x \in \mathcal{T}$, $\tsuc(x)$ is finite, we call $\mathcal{T}$ \emph{locally finite}.
Let $\tssuc(x) = \bigcup_{y \in \tsuc(x)} \tsuc(y)$ be the set of \emph{grandchildren} of $x$.
Let $\tLeaf(\mathcal{T}) = \Setby{x \in T}{\forall y \in T, x \not\prec y}$ be the set of leaves of $\mathcal{T}$.
If $\tLeaf(\mathcal{T}) = \emptyset$ we call $\mathcal{T}$ \emph{leafless}.

The \emph{depth} of $x$ is the ordinal of $\tPre(x)$, which we denote by $\tdepth(x)$.
Let $\theight(\mathcal{T}) = \sup\Setby{\tdepth(x) + 1}{x \in \mathcal{T}}$ be the \emph{height} of $\mathcal{T}$.
For any ordinal $\ell < \theight(\mathcal{T})$, let $\tlevel(\ell) = \Setby{x \in \mathcal{T}}{\tdepth(x) = \ell}$ be the $\ell$-th \emph{level} of $\mathcal{T}$. 

\subsection{A Hodge theory lemma} \label{sapp:hodge-theory}

The following was proved in \cite{PeTa18}.

\begin{theorem} [{\cite[Theorem 3.2]{PeTa18}}] \label{thm:prelimitive-compact-support}
  Let $Z$ be an open subset of $F$ such that $\overline{Z}$ is a submanifold of $F$ with boundary $\partial Z$.
  Then for any $k \in \mathbb{N}$ with $1 \leq k \leq \dim F$ there is an operator  $I_Z^k$ preserving smooth families of $k$-forms
  \begin{equation*}
    \Setby{\xi \in \Omega^k_\cspt(F)}{\support \xi \subset Z, \Res{\xi}_Z \in \der \Omega^{k-1}_\cspt(Z)} \to \Setby{\eta \in \Omega^{k-1}_\cspt(F)}{\support \eta \subset \overline{Z}}
  \end{equation*}
  satisfying $d \circ I_Z^q = \identity$.
\end{theorem}

\begin{proof}
  For the sake of being self-contained, we repeat the proof here. 
  By \cite{MR1893604} we have a weighted Hodge-Laplacian $\Delta_\mu \colon \Omega^k_\cspt(Z) \to \Omega^k_\cspt(Z)$ on $Z$ which is endowed with a specific metric $g$ and measure $\mu$.
  Its Green operator $G_\mu \colon \Omega^k_\cspt(Z) \to \Omega^k(Z)$ and  weighted codifferential $\delta_\mu \colon \Omega^k_\cspt(Z) \to \Omega^{k-1}_\cspt(Z)$ satisfy $\der \circ \delta_\mu \circ G_\mu \circ \der = \der$.
  Any $\eta \in G_\mu(\Omega^k_\cspt(Z))$ has an extension $\tilde{\eta} \in \Omega^k_\cspt(F)$,such that $\support \tilde{\eta} \subset \overline{Z}$ and $\Res{\tilde{\eta}}_Z = \eta$.
  For $\xi \in \Omega^k_\cspt(F)$ supported in $Z$, if $\Res{\xi}_Z \in \der \Omega^{k-1}_\cspt(Z)$ let $I_Z^k(\xi)$ be the extension of $(\delta_\mu \circ G_\mu) (\Res{\xi}_Z)$ to $\Omega^{k-1}_\cspt(F)$.
  Then $\der \circ I_Z^k = \identity$.
  The operator $I_Z^k$ preserves smooth families, this can be seen as follows. 
  Let $U$ be an open subset of some Euclidean space.
  First, the $p$-derivative of a smooth family $\xi_p$, $p \in U$ of compactly supported forms is compactly supported.
  Since $G_\mu$ is an integral operator with a singular kernel one can pass the $p$-derivative through $G_\mu$, so $\partial_p G_\mu \xi_p$ exists and is a smooth form, for each $p \in U$.
  By an analogous proof for higher order derivatives, $G_\mu \xi_p$, with $p \in U$, is a smooth family.
  The map $\delta_\mu$ preserves smooth families since it is a differential operator.
\end{proof}

\subsection{Three auxiliary technical lemmas} \label{sapp:three-technical}

Now we prove an auxiliary lemma which will ensure the smooth dependence of volumes on parameters after distributed into many connected components (for the proof of the main technical tool below \cref{lem:filtration}).
Let $\kappa \colon B' \to B$ be a covering map.
We define the \emph{pullback} $\kappa^*$ and the \emph{pushforward} $\kappa_*$ of functions as follows 
\begin{align*}
  \kappa^* \colon \cont(B; \mathbb{R}) &\to \cont(B'; \mathbb{R}), &\kappa_* \colon \cont(B'; \mathbb{R}) &\to \cont(B; \mathbb{R}); \\
  (\kappa^*u)(p') &= u(\kappa(p')), &(\kappa_*u)(p) &= \sum_{p' \in \kappa^{-1}(p)} u(p').
\end{align*}
If $B$ is connected, recall that $\#\kappa \in \N$ is the number of sheets of $\kappa$.
Then for any $u \in \cont(B; \mathbb{R})$,
\begin{equation} \label{eq:covering-space-equation}
  (\kappa_* \kappa^* u)(p) = \sum_{p' \in \kappa^{-1}(p)} u(\kappa(p')) = \#\kappa \cdot u(p).
\end{equation}

The following generalizes \cite[Lemma 3.7]{PeTa18}.

\begin{lemma} \label{lem:approximation-lemma}
  Let $\kappa \colon B' \to B$ be a covering map with $B'$ compact (so $B$ is compact).
  Let $a \in \cont(B; \mathbb{R})$, $u \in \smth(B; \mathbb{R})$ such that $u < a$. 
  Then for any $a' \in \cont(B'; \mathbb{R})$ with $\kappa_* a' = a$, there is $u' \in \smth(B'; \mathbb{R})$ such that $u' < a'$ and $\kappa_* u' = u$.
\end{lemma}

\begin{proof}
  Without loss of generality we assume $B$ is connected and $u = 0$ otherwise we can deal with each connected component of $B$ one by one and replace $a'$ by $a' - u/\#\kappa$, $u'$ by $u' - u/\#\kappa$.

  Choose $\varepsilon > 0$ such that $\#\kappa \cdot \varepsilon < \min a$.
  Define $h' = a' - \varepsilon$, then $\kappa_* h' = a - \#\kappa \cdot \varepsilon > 0$.
  So $\kappa_* (h')^+ > \kappa_* (h')^- \geq 0$.
  Here $(h')^+(p) = \min\Set{h'(p), 0}$ and $(h')^-(p) = \min\Set{-h'(p), 0}$ denote the positive and negative parts of $h'$, respectively.
  Since $h'$ is bounded from below we set $c = \max \kappa_* (h')^- > 0$.
  Define $w' = \frac{(h')^+}{\kappa^* \kappa_* (h')^+} \kappa^* \kappa_* (h')^- -(h')^-$, then $\kappa_*(w') = 0$.
  Moreover,
  \begin{equation*}
    h' - w' = (h')^+ - \frac{\kappa^* \kappa_* (h')^-}{\kappa^* \kappa_*(h')^+} (h')^+ \geq 0.
  \end{equation*}

  By Whitney Approximation Theorem there is $v' \in \smth(B'; \mathbb{R})$ such that $\abs{v' - w'} < \varepsilon / 2$.
  Then let $u' = v' - \frac{1}{\# \kappa} \kappa^* \kappa_*(v') \in \smth(B'; \mathbb{R})$.
  So $\abs{u' - w'} < \varepsilon$, and $\kappa_* u' = 0$ by \cref{eq:covering-space-equation}, hence $a' - u' > h' - w' \geq 0$ is as required.
\end{proof}

Finally, in the paper we need the following elementary result:

\begin{lemma}[{\cite[Lemma~2.1]{PeTa18}}] \label{lem:connecting}
  Let $X$ be a locally connected locally compact Hausdorff space. Let $\mathcal{K}(X)$ be the collection of compact subsets of $X$.
  Let $K \in \mathcal{K}(X)$ and let $A, A' \subset X$ be connected and precompact.
  If $A, A'$ lies in the same connected component $C$ of $X$ then there is $L \in \mathcal{K}(X)$ such that they lie in the same connected component of $L \cap C$.
\end{lemma}

Then \cref{lem:saturating-threshold} follows from \cref{lem:connecting}.

\begin{proof} [Proof of \cref{lem:saturating-threshold}]
  Let $P$ be the fiber of $\lb{A}$, so $\lb{A} = (\Res{\pi}_A, A, B, P, \Res{f}_A)$.
  Let $B_A, P_A$ be manifolds such that $\Rls \lb{A} = (\Rls(\Res{\pi}_A), A, B_A, P_A, f)$.
  Fix $\alpha \in \Regular(f) \cap f(A)$ and then we consider $\lb{A}_{[\alpha]}$.
  Since the fiber of $\lb{A}_{[\alpha]}$ is a precompact submanifold of $F$ with boundary, it can only have finitely many components.
  Let $P_\alpha$ denote the fiber of $\Res{(\Rls \lb{A})}_{A_{[\alpha]}}$, so then $P_\alpha$ has finitely many components.
  By \cref{lem:connecting}, there is $K \in \mathcal{K}(\overline{P_A})$ which is connected and contains both $x_0$ and $P_\alpha$, where the closure of $P_A$ is taken in $F$.
  Suppose $\beta \in \Regular(f)$ and $\beta \geq \max_{i \in \mathcal{I}}\max_K h_i$, then $P_\beta \supset K$.
  Let $\bfiota_{\alpha, \beta} \colon \Res{(\Rls \lb{A})}_{A_{[\alpha]}} \hookrightarrow \Res{(\Rls \lb{A})}_{A_{[\beta]}}$ be the embedding, and let $\Rls \bfiota_{\alpha, \beta} = (\iota_{\alpha, \beta}, \kappa_{\alpha, \beta}).$
  Let $B_{A, \alpha}$ and $B_{A, \beta}$ be the bases of $\Res{(\Rls \lb{A})}_{A_{[\alpha]}}$ and $\Res{(\Rls \lb{A})}_{A_{[\beta]}}$, respectively.
  Then $\kappa_{\alpha, \beta} \colon B_{A, \alpha} \to B_{A, \beta}$ is a covering map.
  Note that a component of $P_\beta$ contains of $P_\alpha$, this means the image of $\kappa_{\alpha, \beta}$ is a one-fold covering of $B_A$.
  Let $\bfiota_\beta \colon \Res{(\Rls \lb{A})}_{A_{[\beta]}} \hookrightarrow \Rls \lb{A}$ be the embedding, and let $\Rls \bfiota_\beta = (\iota_\beta, \kappa_\beta)$, then $\kappa_\beta \colon B_{A, \beta} \to B_A$ is a diffeomorphism.
  For the same reason, $\# \kappa_{\beta'} = 1$ for any $\beta' \in \Regular(f)$ no less than $\beta$.
\end{proof}

\emph{Acknowledgments}. We thank Ioan Benjenaru, Bruce Driver, Alessio Figalli, Rafe Mazzeo, Lei Ni, and Alan Weinstein for very helpful discussions. 
The authors are supported by NSF CAREER DMS-1518420. 
The first author received support from Severo Ochoa Program at ICMAT in Spain.

\bigskip
\bigskip

\noindent
\'Alvaro Pelayo and Xiudi Tang
\\
Department of Mathematics\\
University of California, San Diego\\
9500 Gilman Drive $\#$ 0112\\
La Jolla, CA 92093-0112, USA
\\
\\
E-mail: alpelayo@math.ucsd.edu\\
E-mail: xdtang@ucsd.edu

\end{document}